\documentclass[reqno]{amsart}
\usepackage{amsmath,amssymb,amsthm,mathrsfs,xspace,enumitem}
\usepackage[all]{xy}
\DeclareMathOperator{\ap}{Ass}
\DeclareMathOperator{\ann}{ann}
\DeclareMathOperator{\spec}{Spec}
\DeclareMathOperator{\dep}{depth}
\DeclareMathOperator{\nat}{Nat}
\DeclareMathOperator{\bighom}{Hom}
\DeclareMathOperator{\bigmod}{Mod}
\DeclareMathOperator{\fgmod}{mod}
\DeclareMathOperator{\tor}{Tor}
\DeclareMathOperator{\ext}{Ext}
\DeclareMathOperator{\bigh}{H}
\DeclareMathOperator{\im}{im}
\DeclareMathOperator{\coker}{coker}
\DeclareMathOperator{\gr}{gr}
\DeclareMathOperator{\id}{id}
\DeclareMathOperator{\lcm}{lcm}
\renewcommand{\geq}{\geqslant}
\renewcommand{\leq}{\leqslant}
\renewcommand{\bar}{\overline}

\renewcommand{\phi}{\varphi}
\newcommand{\ds}{\displaystyle}
\newcommand{\mc}{multiplicatively closed\xspace}
\newcommand{\cmc}{common multiplicatively closed\xspace}
\newcommand{\nn}{\mathbb{N}}
\newcommand{\zln}{\mathbb{Z}}
\newcommand{\rlin}{\mathscr{F}}
\newcommand{\seq}{\mathscr{S}}
\newcommand{\tormod}{\mathscr{T}}
\newcommand{\separate}{\medskip}

\theoremstyle{plain}
\newtheorem{thm}{Theorem}
\newtheorem{cor}[thm]{Corollary}

\newtheorem{lem}[thm]{Lemma}

\theoremstyle{definition}
\newtheorem{defn}[thm]{Definition}
\newtheorem{notn}[thm]{Notation}
\newtheorem{eg}[thm]{Example}
\newtheorem*{ack}{Acknowledgements}

\theoremstyle{remark}
\newtheorem{rmk}[thm]{Remark}
\newtheorem{qn}[thm]{Question}

\begin{document}
\title{Covariant functors and asymptotic stability}
\author{Tony Se}
\address{Department of Mathematics\\
University of Kansas\\
 Lawrence, KS 66045-7523 USA}
\email{tonyse@ku.edu}

\date{\today}


\numberwithin{thm}{section}

\begin{abstract}
  Let $R$ be a commutative Noetherian ring, $I,J$ ideals of $R$ and
  $M$ a finitely generated $R$-module. Let $F$ be a covariant $R$-linear functor
  from the category of finitely generated $R$-modules to itself.
  We first show that if $F$ is coherent, then the sets $\ap_R F(M/I^n M)$,
  $\ap_R F(I^{n-1}M/I^n M)$ and the values $\dep_J F(M/I^n M)$, $\dep_J F(I^{n-1}M/I^n M)$
  become independent of $n$ for large $n$.
  Next, we consider several examples in which $F$ is a rather familiar functor, but is
  not coherent or not even finitely generated in general. In these cases, the sets
  $\ap_R F(M/I^n M)$ still become independent of $n$ for large $n$. We then show one
  negative result where $F$ is not finitely generated. Finally, we give a positive result
  where $F$ belongs to a special class of functors which are not finitely generated in
  general, an example of which is the zeroth local cohomology functor.
\end{abstract}

\maketitle

\section{Introduction}

In this paper, we will extend two results on asymptotic stability
by M.\ Brodmann. Let us begin by fixing some terminology.
A ring will mean a commutative ring with unity, unless specified otherwise.
For a ring $R$, we let
$\bigmod(R)$ denote the category of $R$-modules and $\fgmod(R)$
the category of finitely generated $R$-modules.
A functor will mean a covariant functor. For a nonempty set $X$ and
a sequence of elements $\{x_n\}_{n \geq k}$ of $X$,
we say that asymptotic stability holds for the elements $x_n$, or that
the elements $x_n$ stabilize, if the sequence $\{x_n\}_{n \geq k}$ is
eventually constant.
\separate

For the rest of this section, we will let $R$ be a Noetherian ring unless specified
otherwise, $L,M,N \in \fgmod(R)$ and $I,J$ be ideals of $R$.
The background of our project can be traced back to
one of Ratliff's papers.

\begin{qn} \cite[Introduction]{R}
  Suppose that $R$ is a domain and $P$ is a prime ideal of $R$.
  If $P \in \ap_R(R/I^k)$ for some $k \geq 1$, is
  $P \in \ap_R (R/I^n)$ for all large $n$ ?
\end{qn}

Brodmann \cite[(9)]{BPrime} gave a negative answer to
the question, but at the same time, he proved a related, by now
well-known result. Using the notation established so far,
we will state his\footnote{Prof.~Daniel Katz informed the author that
Brodmann's proof is already present in \cite{R}, although
arranged in a different order.} first result that we are interested in. 

\begin{thm} \cite[page~16]{BPrime}
\label{thm:BPrime}
The sets $\ap_R(M/I^n M)$ and $\ap_R(I^{n-1} M/I^n M)$ stabilize.
\end{thm}

The second result that we are interested in is as follows.

\begin{thm} \cite[Theorems~2(i) and 12(i)]{BSpread}
\label{thm:BSpread}
The values $\dep_J(M/I^n M)$ and
$\dep_J(I^{n-1} M/I^n M)$ stabilize.
\end{thm}

Most of this paper will be related to Theorem~\ref{thm:BPrime}.
There have been numerous generalizations of the theorem
over the years. Here are a few of them\footnote{Although the
theorems quoted here are related, the authors of \cite{KW} and \cite{MS}
did not seem to know about the results of each other.}.

\begin{thm} \cite[Theorem~1]{MS}
  The sets $\ap_R \tor_i^R (N, R/I^n)$
  and $\ap_R \tor_i^R (N,I^{n-1}/I^n)$ stabilize for any $i \geq 0$.
\end{thm}

\begin{thm} \cite[Proposition~3.4]{KW}
  \label{thm:KWhom}
  Let $L \xrightarrow{\alpha} M \xrightarrow{\beta} N$ be a complex.
  Suppose that $L' \subseteq L$, $M' \subseteq M$ and $N' \subseteq N$
  are submodules such that $\alpha(L') \subseteq M'$ and
  $\beta(M') \subseteq N'$. For $n \geq 0$, let $\bigh(n)$ denote
  the homology of the induced complex
  \[
    \frac{L}{I^n L'} \xrightarrow{\alpha_n}
    \frac{M}{I^n M'} \xrightarrow{\beta_n}
    \frac{N}{I^n N'}
  \]
  Then the sets $\ap_R \bigh(n)$ stabilize.
\end{thm}

\begin{cor} \cite[Corollary~3.5]{KW} \label{cor:KW}
  Let $M' \subseteq M$ be a submodule. Then for any $i \geq 0$,
  the sets $\ap_R \tor_i^R (N, M/I^n M')$ and
  $\ap_R \ext_R^i (N, M/I^n M')$ stabilize.
\end{cor}

A rather extensive introduction to results related to
Theorems~\ref{thm:BPrime} and \ref{thm:BSpread} can be found
in \cite{McA}. However, we will proceed in a different direction.
Our main goal is to relate the theorems to the following notions.

\begin{notn}
Let $R$ be a commutative ring and $M \in \bigmod(R)$. Then we let $h_M$
denote the functor $\bighom_R(M,-)$.
We let $\rlin$ denote the category of $R$-linear covariant functors
$F$ from $\fgmod(R)$ to itself.
\end{notn}

\begin{defn} \cite[page 53]{Har} \label{defn:coh}
Let $R$ be a Noetherian ring and $F \in \rlin$. We say that:
\begin{enumerate}[label=(\arabic*),align=left,leftmargin=*,nosep]
  \item $F$ is representable
  if $F \cong h_M$ for some
  $M \in \fgmod(R)$;
  \item $F$ is coherent if
  there exist $M,N \in \fgmod(R)$
  and an exact sequence
  $h_N \to h_M \to F \to 0$;
  \item $F$ is finitely generated
  if there exist $M \in \fgmod(R)$
  and an exact sequence
  $h_M \to F \to 0$.
\end{enumerate}
\end{defn}

\begin{rmk}
  Representable $\Rightarrow$
  coherent $\Rightarrow$
  finitely generated $\Rightarrow$
  $R$-linear
\end{rmk}

We can now state our main result, which will be proved in several steps in
Section~\ref{sec:stab}.

\begin{thm} \label{thm:main}
  Let $R$ be a Noetherian ring, $I,J$ ideals of $R$,
  $M \in \fgmod(R)$ and $F$ be a coherent functor.
  Then the sets $\ap_R F(M/I^n M)$, $\ap_R F(I^{n-1}M/I^n M)$
  and the values $\dep_J F(M/I^n M)$, $\dep_J F(I^{n-1}M/I^n M)$ stabilize.
\end{thm}

\begin{rmk}
  Theorem~\ref{thm:main} gives an extension of Theorem~\ref{thm:KWhom} in the following sense.
  Using the notation in Theorem~\ref{thm:KWhom}, let $L=L'$, $M=M'$ and $N=N'$.
  Then Theorem~\ref{thm:KWhom} is an instance of Theorem~\ref{thm:main} by
  Lemma~\ref{lem:Har}\ref{item:complex} (cf.\ proof of Theorem~\ref{thm:midfin}).
  However, by \cite[Example 5.5]{Har}, not all coherent functors are of the form given by
  Lemma~\ref{lem:Har}\ref{item:complex}. A technical generalization of Theorem~\ref{thm:KWhom}
  is given by Corollary~\ref{cor:KWhom}.
\end{rmk}

A summary of the rest of the paper is as follows. In Section~\ref{sec:twoeg}, we consider
two covariant $R$-linear functors, the zeroth local cohomology functor $\Gamma_I$
where $I$ is an ideal of $R$, and the torsion functor $\tau_S$ where $S$ is a multiplicatively
closed subset of $R$. We show that in most cases, the functors $\id\!/\Gamma_I$ and
$\id\!/\tau_S$ are finitely generated but not coherent, while the functors $\Gamma_I$
and $\tau_S$ are not even finitely generated. However, if $F = \id\!/\Gamma_I$,
$\id\!/\tau_S$, $\Gamma_I$ or $\tau_S$, then whether or not $F$ is coherent, the sets
$\ap_R F(M/I^n M)$ and $\ap_R F(I^{n-1}M/I^n M)$ always stabilize.
In Section~\ref{sec:dd}, we consider the case
where $R$ is a Dedekind domain. We show that if $F$ is a finitely generated functor,
then the sets $\ap_R F(M/I^n M)$ stabilize. We give a family of non-finitely generated
functors $F$ such that the sets $\ap_R F(M/I^n M)$ do not stabilize. In Section~\ref{sec:midfin},
we consider a complex $\seq \colon A \to B \to C$ of $R$-modules where $B \in \fgmod(R)$
and the functor $F(-) = \bigh (\seq \otimes -)$, an example of which is the zeroth local
cohomology functor. We show that if $R$ is a one-dimensional Noetherian domain, then
the sets $\ap_R F(M/I^n M)$ stabilize.

\begin{ack}
  This paper forms part of the author's thesis at the University of Kansas.
  The author would like to express his gratitude to his advisor,
  Prof.\ Hailong Dao, for suggesting the topic of this paper and overseeing the research
  project, and to Prof.\ Daniel Katz, who filled in the historical details of this paper.
  The author also thanks Arindam Banerjee and William Sanders for conversations that enriched
  this project.
\end{ack}

\section{Proof of stability results}
\label{sec:stab}

In the section, we let $R$ be a Noetherian ring. All $R$-modules
will be finitely generated unless specified otherwise.
We will prove our main result, Theorem~\ref{thm:main}, which will follow from
Corollary~\ref{cor:coh}, Corollary~\ref{cor:dep} and Corollary~\ref{cor:graded}.
First, we need a slightly more general result than Theorem~\ref{thm:KWhom}.
We recall that the Theorem follows from an
even more general result.

\begin{thm} \cite[Proof of Proposition~3.4]{KW}
  \label{thm:KWgen}
  Let $I \subseteq R$ be an ideal, $T \in \fgmod(R)$ and
  $U,V,W$ submodules of $T$ such that $W \subseteq V$. Then
  the sets $\ap_R ((U+I^n V)/I^n W)$ stabilize.
\end{thm}

\begin{cor} \label{cor:KWhom}
  Consider the situation as in Theorem~\ref{thm:KWhom}.
  Let $c \in \nn$ and $L_1,L_2$ be submodules of $L$ such that
  $I^c L' \subseteq L_2$. For $n \geq c$, let $\bigh(n)$ denote
  the homology of the induced complex
  \[
    \frac{L_1 + I^{n-c}L_2}{I^n L'} \xrightarrow{\alpha_n}
    \frac{M}{I^n M'} \xrightarrow{\beta_n}
    \frac{N}{I^n N'}
  \]
  Then the sets $\ap_R \bigh(n)$ stabilize.
\end{cor}

\begin{proof}
  We follow \cite[Proof of Proposition~3.4]{KW}. By the Artin-Rees
  Lemma, there is $d \geq c$ such that for all $n \geq d$,
  $\beta(M) \cap I^n N' = I^{n-d}(\beta(M) \cap I^d N')$.
  Then for $n \geq d$, we have
  \begin{align*}
    \bigh(n) &= \frac{\ker(\beta_n)}{\im(\alpha_n)}\\
    &= \frac{\ker(\beta) + I^{n-d} (\beta^{-1}(I^d N'))}
    {\alpha(L_1) + I^{n-d}(I^{d-c}\alpha(L_2) + I^d M')}\ .
  \end{align*}
  The result then follows from Theorem~\ref{thm:KWgen} by letting
  \[
  \begin{aligned}[b]
    T &= \frac{M}{\alpha(L_1)}\ ,\\
    V &= \frac{\beta^{-1}(I^d N') + \alpha(L_1)}{\alpha(L_1)}
  \end{aligned}
  \qquad \text{and} \qquad
  \begin{aligned}[b]
    U &= \frac{\ker(\beta)}{\alpha(L_1)}\ ,\\
    W &= \frac{I^{d-c}\alpha(L_2) + I^d M' + \alpha(L_1)}{\alpha(L_1)}\ .
  \end{aligned} \qedhere
  \]
\end{proof}

Next, we recall some results from \cite{Har}.

\begin{lem} \label{lem:Har}
  \cite[Lemma~1.2, Examples~2.1--2.5]{Har}
  \begin{enumerate}[label=(\alph*),align=left,leftmargin=*,nosep]
  \item For any $M \in \fgmod(R)$ and $F \in \rlin$,
  there is a natural isomorphism $\nat_{\rlin}(h_M,F) \cong F(M)$
  given by $T \mapsto T_M(\id_M\!)$. \label{item:nat}
  \item Let $P_{\bullet}$ be a complex of finitely generated $R$-modules. Then
  for any $i \in \zln$, the functor $\bigh_i (P_{\bullet} \otimes -)$ is coherent.
  \label{item:complex}
  \item Let $M \in \fgmod(R)$. Then for any $i \geq 0$, the functors
  $\tor_i^R (M, -) \text{ and } \ext_R^i (M, -)$ are coherent.
  \label{item:torext}
  \end{enumerate}
\end{lem}
\vspace{-12pt}

We then obtain the following generalization of the first half
of Theorem~\ref{thm:BPrime}. By Lemma~\ref{lem:Har}\ref{item:torext},
Corollary~\ref{cor:coh} may also be viewed as a generalization of
Corollary~\ref{cor:KW}.

\begin{cor} \label{cor:coh}
  Let $F$ be a coherent functor, $M \in \fgmod(R)$, $M'$ be a submodule
  of $M$ and $I \subseteq R$ an ideal.
  Then the sets $\ap_R F( M/I^n M' )$ stabilize.
\end{cor}

\begin{proof}
  Let $F$ be given by $h_L \to h_K \to F \to 0$. By
  Lemma~\ref{lem:Har}\ref{item:nat},
  the map $h_L \to h_K$ arises from a map $f \colon K \to L$.
  Choose free resolutions of $K$ and $L$ and a lift of $f$ such that
  the following diagram commutes.
  
  \begin{equation} \label{eqn:lift}
    \begin{split}
    \xymatrix{
      R^{\oplus k_1} \ar[r] \ar[d]^{\beta} & R^{\oplus \ell_1} \ar[d]^{\gamma}\\
      R^{\oplus k_0} \ar[r]^{\alpha} \ar[d]  & R^{\oplus \ell_0} \ar[d]\\
      K \ar[r]^f \ar[d] & L \ar[d] \\
      0          & 0
    }
    \end{split}
  \end{equation}
  Apply $\bighom_R(-,M/I^n M')$ to get the commutative diagram
  
  \[
    \xymatrix{
    \ds \frac{M^{\oplus \ell_1}}{I^n \left((M')^{\oplus \ell_1}\right)}
    \ar[r] &
    \ds \frac{M^{\oplus k_1}}{I^n \left((M')^{\oplus k_1}\right)}\\
    \ds \frac{M^{\oplus \ell_0}}{I^n \left((M')^{\oplus \ell_0}\right)}
    \ar[u]_{\gamma_n^{*}} \ar[r]^{\alpha_n^{*}} &
    \ds \frac{M^{\oplus k_0}}{I^n \left((M')^{\oplus k_0}\right)}
    \ar[u]_{\beta_n^{*}}\\
    h_L\left( \ds \frac{M}{I^n M'} \right)
    \ar[u] \ar[r]^{f_n^{*}} &
    h_K\left( \ds \frac{M}{I^n M'} \right)
    \ar[r] \ar[u] &
    F \left( \ds \frac{M}{I^n M'} \right)
    \ar[r] & 0\\
    0 \ar[u] &  0 \ar[u]
    }
  \]
  where $f_n^{*},\alpha_n^{*},\beta_n^{*},\gamma_n^{*}$ are
  induced by $f,\alpha,\beta,\gamma$ respectively. Then we have
  \[
    F\left(\frac{M}{I^n M'}\right)
    \cong \frac{\ker \beta_n^{*}}
          {\alpha_n^{*} \left( \ker \gamma_n^{*} \right)}\ .
  \]
  Similarly, we apply $\bighom_R(-,M)$ to (\ref{eqn:lift})
  to get maps $\alpha^{*},\beta^{*},\gamma^{*}$
  induced by $\alpha,\beta,\gamma$ respectively. Let
  $A = M^{\oplus \ell_0}$, $A' = (M')^{\oplus \ell_0}$ and
  $B' = (M')^{\oplus \ell_1}$. As in the
  proof of Corollary~\ref{cor:KWhom}, there is $c \in \nn$
  such that $\gamma^{*}(A) \cap I^n B'
  = I^{n-c} (\gamma^{*}(A) \cap I^c B')$
  for all $n \geq c$, and hence
  \[
    \ker (\gamma_n^{*}) =
    \frac{\ker (\gamma^{*}) + I^{n-c}
    \left( \left( \gamma^{*} \right)^{-1} (I^c B') \right)}
    {I^n A'}\ .
  \]
  The result then follows from Corollary~\ref{cor:KWhom}.
\end{proof}

We next generalize the first half of Theorem~\ref{thm:BSpread}
along similar lines.

\begin{notn}
Let $T,U,V,W$ be as in Theorem~\ref{thm:KWgen}.
We let $T_n = (T,U,V,W)_n = (U+I^n V)/I^n W$.
\end{notn}

\begin{rmk}
Let $L$ be an ideal of $R$. For a submodule $S$ of $T$, we let
$\bar{S}$ be the image of $S$ under the natural projection
$T \to T/LU$. Then we have
\begin{align*}
  \frac{T_n}{L T_n}
  &= \frac{U + I^n V}{LU + LI^n V + I^n W}\\
  &= \frac{\bar{U} + I^n \bar{V}}{LI^n \bar{V} + I^n \bar{W}}\\
  &= (\,\bar{T},\bar{U},\bar{V},\bar{LV+W}\,)_n
\end{align*}
\end{rmk}

\begin{thm} \label{thm:Bdepth}
The values $\dep_J T_n$ stabilize.
\end{thm}

\begin{proof}
First, suppose that $T_n/JT_n= (\,\bar{T},\bar{U},\bar{V},\bar{JV+W}\,)_n
= 0$ for infinitely many $n$. Then by Theorem~\ref{thm:KWgen},
we see that $\ap_R \bar{T}_n = \emptyset$ for large $n$. So
for all large $n$, we have $T_n/JT_n = 0$ and hence $\dep_J T_n
= \infty$. Hence we may assume that $T_n \neq JT_n$ for large $n$.
\separate

The rest of the proof is the same as that in \cite[Theorem~2(i)]{BSpread}.
We let $h_T = \liminf_{n \to \infty} \dep_J(T_n)$, $\ell_T =
\lim_{n \to \infty} \dep_J(T_n)$ if such exists, and prove by induction
on $h_T$ that $\ell_T = h_T$. Suppose that $h_T = 0$. Then $J \subseteq
\{ r \in P \mid P \in \ap_R T_n \}$ for infinitely many $n$.
By Theorem~\ref{thm:KWgen}, we have $J \subseteq \{ r \in P \mid P \in
\ap_R T_n \}$ for all large $n$, so $\ell_T = h_T = 0$.
\separate

Now suppose that $h_T > 0$. Then by Theorem~\ref{thm:KWgen}, there
is $x \in J$ such that $x \notin \{ r \in P \mid P \in \ap_R T_n\}$
for all large $n$. Writing $T_n/xT_n = (\, \bar{T},
\bar{U},\bar{V},\bar{xV + W}\,)_n$, we have
$\dep_J \bar{T}_n = \dep_J T_n - 1$ for all large $n$.
Hence $h_{\bar{T}} = h_T - 1$. By induction, we have
$\ell_{\bar{T}} = h_{\bar{T}}$, so $\ell_T = \ell_{\bar{T}} + 1
= h_T$.
\end{proof}

\begin{cor}
  Let $J \subseteq R$ be an ideal. Consider the situation as in
  Corollary~\ref{cor:KWhom} with the complexes
  \[
    \frac{L_1 + I^{n-c}L_2}{I^n L'} \xrightarrow{\alpha_n}
    \frac{M}{I^n M'} \xrightarrow{\beta_n}
    \frac{N}{I^n N'}
  \]
  and $\bigh(n)$ denoting the homology of the complex. Then the
  values $\dep_J \bigh(n)$ stabilize.
\end{cor}

\begin{cor} \label{cor:dep}
  Let $F$ be a coherent functor, $M \in \fgmod(R)$, $M'$ be a submodule
  of $M$ and $I,J$ be ideals of $R$.
  Then the values $\dep_J F( M/I^n M' )$ stabilize.
\end{cor}

In order to generalize the rest of Theorems~\ref{thm:BPrime}
and \ref{thm:BSpread}, we let $S = \bigoplus_{n \geq 0} R_n$
be a Noetherian $R$-algebra generated in degree 1 with
$R_0 = R$. We will use a result from \cite{MS}.

\begin{thm} \label{thm:MSgraded}
\cite[Lemma~2.1]{MS}
Let $M = \bigoplus_{n \in \zln} M_n$ be a finitely generated
graded $S$-module. Then the sets $\ap_R M_n$ stabilize.
\end{thm}

\begin{cor} \label{cor:gradedap}
Let $L \to M \to N$ be a complex of $\zln$-graded $S$-modules,
where the maps are homogeneous and $M \in \fgmod(S)$. Let
$H = \bigoplus_{\zln} H_n$ be the homology of the complex. Then
the sets $\ap_R H_n$ stabilize.
\end{cor}

\begin{cor} \label{cor:gradeddep}
Let $M = \bigoplus_{n \in \zln} M_n$ be a finitely generated
graded $S$-module, for example the module $H$ as in
Corollary~\ref{cor:gradedap}. Let $J$ be an ideal of $R$.
Then the values $\dep_J M_n$ stabilize.
\end{cor}

\begin{proof}
The proof of Theorem~\ref{thm:Bdepth} works, since
$M/JM = \bigoplus_{n \in \zln} (M_n / J M_n)$ and
$M/x M$ are also finitely generated graded $S$-modules.
\end{proof}

\begin{cor} \label{cor:graded}
  Let $I$ be an ideal of $R$, $M \in \fgmod(R)$ and $M' \subseteq M$
  be a submodule. Then the sets $\ap_R F( I^n M / I^n M')$ and the values
  $\dep_J F(I^n M/I^n M')$ stabilize.
\end{cor}

\begin{proof}
  As in Corollary~\ref{cor:coh}, we apply $\bighom_R(-,I^n M/I^n M')$
  to (\ref{eqn:lift}) to get
  \[
  \xymatrix{
    \ds \frac{I^n \left(M^{\oplus \ell_1}\right)}
        {I^n \left((M')^{\oplus \ell_1}\right)}
    \ar[r] &
    \ds \frac{I^n \left(M^{\oplus k_1}\right)}
        {I^n \left((M')^{\oplus k_1}\right)}\\
    \ds \frac{I^n \left(M^{\oplus \ell_0}\right)}
        {I^n \left((M')^{\oplus \ell_0}\right)}
    \ar[u]_{\gamma_n^*} \ar[r]^{\alpha_n^*} &
    \ds \frac{I^n \left(M^{\oplus k_0}\right)}
        {I^n \left((M')^{\oplus k_0}\right)}
    \ar[u]_{\beta_n^*}\\
    h_L\left( \ds \frac{I^n M}{I^n M'} \right)
    \ar[u] \ar[r]^{f_n^*} &
    h_K\left( \ds \frac{I^n M}{I^n M'} \right)
    \ar[u] \ar[r] &
    F \left( \ds \frac{I^n M}{I^n M'} \right)
    \ar[r] & 0\\
    0 \ar[u] & 0 \ar[u]
  }
  \]
  Again we have $\ds F\left(\frac{I^n M}{I^n M'}\right) \cong
  \frac{\ker \beta_n^{*}}{\alpha_n^{*} \left( \ker \gamma_n^{*} \right)}$\ ,
  where $\alpha_n^*,\beta_n^*,\gamma_n^*$ are the maps induced by
  $\alpha,\beta,\gamma$ in (\ref{eqn:lift}) respectively, so the result
  follows from Corollaries~\ref{cor:gradedap} and \ref{cor:gradeddep} by letting
  $S = \bigoplus_{n \geq 0} I^n$.
\end{proof}

A coherent functor $F$ given by $h_L \to h_K \to F \to 0$ can be
considered as a functor $\bigmod(R) \to \bigmod(R)$ since $h_L$ and $h_K$
are (cf.\ \cite[Remark 3.3]{Har}). So the proof of Corollary~\ref{cor:graded}
gives the next result.

\begin{cor}
  Let $F$ be a coherent functor, $M \in \fgmod(R)$, $M' \subseteq M$
  be a submodule, $I$ be an ideal of $R$, $S = \mathscr{R}(I)
  = \bigoplus_{n \geq 0} I^n$ and $\gr(I) = \bigoplus_{n \geq 0} I^n/I^{n+1}$.
  Then:
  \begin{enumerate}[label=(\alph*),align=left,leftmargin=*,nosep]
    \item \label{item:rees}
    $F\left( \bigoplus_{n \geq 0} I^n M / I^n M' \right)
    = \bigoplus_{n \geq 0} F(I^n M / I^n M')$ is a finitely generated
    graded $S$-module.
    \item \label{item:gri}
    When $M' = IM$, $F\left( \bigoplus_{n \geq 0} I^n M / I^{n+1} M \right)
    = \bigoplus_{n \geq 0} F(I^n M / I^{n+1} M)$ is a finitely generated
    graded $\gr(I)$-module.
    \item The module structures over $S$ and $\gr(I)$ in \ref{item:rees}
    and \ref{item:gri} respectively correspond to the multiplication maps
    given by applying $F$ to $I^n M / I^n M' \xrightarrow{x}
    I^{n+m} M / I^{n+m} M'$, where $x \in I^m$.
  \end{enumerate}
\end{cor}

\begin{rmk}
  Instead of studying asymptotic stability properties of covariant
  coherent functors, one may want to consider contravariant coherent functors
  as well. Unfortunately, as stated in \cite[Remark 3.6]{KW}, the sets
  $\ap_R \ext_R^i (R/I^n,R)$ do not stabilize in general, so our
  main focus will be on covariant functors. See \cite[Introduction]{Si} and
  \cite[Proposition 2.1]{SS} for further details.
\end{rmk}

\section{Examples of non-coherent functors with asymptotic stability}
\label{sec:twoeg}

In view of the results in Section~\ref{sec:stab}, one may be interested
in knowing whether or not a $R$-linear covariant functor is coherent.
Some important examples of coherent functors are given in Lemma~\ref{lem:Har}.
In this section, we will study the zeroth local cohomology functor $\Gamma_I = \bigh_I^0$
where $I$ is an ideal of $R$, and the torsion functor $\tau_S$ where $S$ is a multiplicatively
closed subset of $R$. It turns out that if $F = \Gamma_I$, $\tau_S$, $\id\!/\Gamma_I$
or $\id\!/\tau_S$, then the functor $F$ is usually not coherent. However, we will
see that whether or not $F$ is coherent, the sets $\ap_R F(M/I^n M)$ and
$\ap_R F(I^{n-1}M / I^n M)$ always stabilize.

\begin{lem}[Yoneda's Lemma] \label{lem:Yon}
  Let $R$ be a Noetherian ring and $F$ be a finitely generated functor
  given by $h_M \xrightarrow{T} F \to 0$. Then for any $N \in \fgmod(R)$
  and $x \in F(N)$, there is $f \in \bighom_R(M,N)$ such that
  $x = (F(f) \circ T_M)(\id_M)$. In particular, $x \in \im F(f)$.
\end{lem}

\begin{proof}
  If $x \in F(N)$, then we let $f \in \bighom_R(M,N)$ be such that
  $T_N(f)=x$. The result follows from the commutative diagram
  \begin{gather*}
    \xymatrix{
      \bighom_R(M,M) \ar[r]^-{T_M} \ar[d]_{h_M(f)}
      & F(M) \ar[r] \ar[d]^{F(f)} & 0\\
      \bighom_R(M,N) \ar[r]^-{T_N} & F(N) \ar[r] & 0
    }\\[-18pt]
    \qedhere
  \end{gather*}
\end{proof}

\begin{cor} \label{cor:dirlim}
  Let $R$ be a Noetherian ring and $\{ F_{\lambda} \}_{\lambda \in
  \Lambda}$ be a direct system of functors in $\rlin$.
  Let $F = \varinjlim_{\lambda \in \Lambda} F_{\lambda}$ be given by
  $\{\, T_{\lambda} \colon F_{\lambda} \to F \,\}_{\lambda \in \Lambda}$.
  If $F \in \rlin$ and is finitely generated, then
  $F = \im T_{\lambda_0}$ for some $\lambda_0 \in \Lambda$.
  In particular, if $T_{\lambda}$
  is injective for all $\lambda \in \Lambda$, then $F = F_{\lambda}$
  for all $\lambda \geq \lambda_0$.
\end{cor}

\begin{proof}
  Let $F$ be given by $h_M \to F \to 0$. Since $F(M) \in \fgmod(R)$,
  there is $\lambda_0 \in \Lambda$ such that $F(M) = \im (T_{\lambda_0})_M$.
  Let $N \in \fgmod(R)$ and $x \in F(N)$. By Lemma~\ref{lem:Yon},
  there is $f \in \bighom_R(M,N)$ such that $x \in \im F(f)
  \subseteq \im (T_{\lambda_0})_N$.
  \[
    \xymatrix@C=48pt{
      F_{\lambda_0}(M) \ar@{->>}[r]^-{(T_{\lambda_0})_M} \ar[d]_{F_{\lambda_0}(f)}
      & F(M) \ar[d]^{F(f)}\\
      F_{\lambda_0}(N) \ar[r]^-{(T_{\lambda_0})_N}
      & F(N)
    }
  \]
  Therefore $F = T_{\lambda_0}(F_{\lambda_0})$.
\end{proof}

In the following, we will consider two applications of
Corollary~\ref{cor:dirlim}.

\begin{cor} \label{cor:lc}
  Let $I$ be an ideal of a Noetherian ring $R$. The following are equivalent:
  \begin{enumerate}[label=(\alph*),align=left,leftmargin=*,nosep]
    \item $\Gamma_I$ is representable.
    \item $\Gamma_I$ is finitely generated.
    \item $I^n = I^{n+1}$ for some $n \geq 0$.
  \end{enumerate}
\end{cor}

\begin{proof}
  For all $M \in \bigmod(R)$, we have $\Gamma_I(M) = \varinjlim_n
  \bighom_R(R/I^n,M) = \varinjlim_n (0:_M I^n)$. So by
  Corollary~\ref{cor:dirlim}, $\Gamma_I$ is finitely generated iff
  there exists $n \geq 0$ such that $\Gamma_I(M) = \bighom_R(R/I^n,M)$
  for all $M \in \fgmod(R)$ iff $I^n = I^{n+1}$ for some $n \geq 0$
  by considering $M = R/I^{n+1}$ for ``only if''.
\end{proof}

The relationship between our result and Section~\ref{sec:stab} is as follows.

\begin{thm}\cite[Theorem 1.1(a)]{Har}
  Let $F,G$ be coherent functors and $T \colon F \to G$ be a natural
  transformation. Then $\ker(T)$, $\coker(T)$ and $\im(T)$ are also coherent.
\end{thm}

\begin{lem}
  Let $R$ be a Noetherian ring, $I \subseteq R$ be an ideal and $M \in \bigmod(R)$.
  Then $\ap_R \Gamma_I(M) = \ap_R(M) \cap V(I)$ and
  $\ap_R (M/\Gamma_I(M)) = \ap_R(M) \setminus V(I)$, where $V(I) =
  \{ P \in \spec(R) \mid P \supseteq I \}$.
\end{lem}

\begin{cor}
  Let $\{M_n\}_{n \geq 0}$ be a sequence of modules in $\fgmod(R)$ such that
  the sets $\ap_R(M_n)$ stabilize. Let $I \subseteq R$ be an ideal. If
  $I^n \neq I^{n+1}$ for any $n$, then the functor $\id\!/\Gamma_I$ is finitely
  generated but not coherent, and $\Gamma_I$ is not finitely generated.
  However, whether or not $I^n = I^{n+1}$ for any $n$, the sets
  $\ap_R(M_n/\Gamma_I(M_n))$ and $\ap_R \Gamma_I(M_n)$ always stabilize.
\end{cor}

Now we consider our second example.

\begin{lem} \label{lem:cmc}
  Let $R$ be a ring, possibly noncommutative, with 1. Let $S \subseteq R$ and
  $f \colon S \times S \to R$ be a function. The following are equivalent:
  \begin{enumerate}[label=(\alph*),align=left,leftmargin=*,nosep]
    \item For every $r,s \in S$, left $R$-module $M$ and $m \in M$, if $rm = 0$,
    then $f(r,s)m = f(s,r)m = 0$. \label{item:anymod}
    \item For every $r,s \in S$ we have $f(r,s) \in Rr \cap Rs$.
    \label{item:justr}
  \end{enumerate}
\end{lem}

\begin{proof}
  \ref{item:anymod} $\Rightarrow$ \ref{item:justr}:
  Let $r,s \in S$ and $M = R/Rr$. Then $r\bar{1} = \bar{0}$. By assumption,
  we have $f(r,s)\bar{1} = \bar{0}$, so $f(r,s) \in Rr$. Similarly,
  with $M = R/Rs$ we have $f(r,s) \in Rs$, so that $f(r,s) \in Rr \cap Rs$.
  
  \ref{item:justr} $\Rightarrow$ \ref{item:anymod}:
  Let $r,s \in S$ and $m \in M$. By assumption, $f(r,s), f(s,r) \in Rr$. So
  if $rm = 0$, then $f(r,s)m,f(s,r)m \in Rrm = 0$.
\end{proof}

\begin{eg}
  Let $R$ be a UFD, $S = R$ and $f \colon R \times R \to R$.
  Then $f$ satisfies the conditions in Lemma~\ref{lem:cmc} iff
  for all $r,s \in R$ we have $f(r,s) \in (\lcm(r,s))$.
\end{eg}

\begin{defn} Let $R$ be a commutative ring with 1.
  \begin{enumerate}[label=(\arabic*),align=left,leftmargin=*,nosep]
    \item We say that a subset $S \subseteq R$ is
    common multiplicatively closed if $S \neq \emptyset$ and there is
    a function $f \colon S \times S \to S$ satisfying any condition in
    Lemma~\ref{lem:cmc}, or equivalently, for any $r,s \in S$ there is
    $f(r,s) \in S$ that satisfies any condition in Lemma~\ref{lem:cmc}.
    \item We say that a (nonempty) subset $S \subseteq R$ is coprincipal
    if there is $s \in S$ such that $s \in \bigcap_{r \in S} Rr$.
    Such an $s$ is called a cogenerator of $S$.
    \item For any $S \subseteq R$ and $M \in \bigmod(R)$, we let
    $\tau_S(M) = \{ m \in M \mid rm = 0 \text{ for some } r \in S \}$.
    If $S$ is \cmc, then $\tau_S(M)$ is a submodule of $M$.
  \end{enumerate}
\end{defn}

\begin{eg}
  \begin{enumerate}[label=(\arabic*),align=left,leftmargin=*,nosep]
    \item Any singleton subset of $R$ is common multiplicatively closed.
    \item In general, any coprincipal subset $S \subseteq R$ is common
    multiplicatively closed, since if $s \in S$ is a cogenerator,
    then we can let $f(r,t) = s$ for all $r,t \in S$.
    \item Conversely, if $S = \{ s_1, \dots, s_n \} \subseteq R$ is common
    multiplicatively closed, then $S$ has a cogenerator
    $f( \cdots f(f(s_1,s_2),s_3),\dots,s_n)$.
    \item Any multiplicatively closed subset of $R$ is common
    multiplicatively closed.
    \item If $r,s \in \zln$ and $(0) \neq (s) \subsetneq (r)$, then the
    subset $\{r,s\}$ of $\zln$ is common multiplicatively closed and
    coprincipal but not multiplicatively closed.
    \item Let $a \in \zln$ such that $a \neq 0,\pm 1$. Let $S = \{ a^2 \}
    \cup \{ a^{8+12n} \mid n \geq 0 \}$. Then $S$ is a common multiplicatively
    closed subset of $\zln$ by the function $f(s,t) = (st)^2$,
    and $S$ is neither multiplicatively closed nor coprincipal.
    \item Let $a \in \zln$ such that $a \neq 0,\pm 1$. Then the infinite
    multiplicatively closed subset $S=\{ a^{-n} \mid n \geq 0 \}$ of $\zln_a$
    is coprincipal with 1 as a cogenerator; the subset $\{ a^n \mid n \geq 0 \}$
    of $\zln$ is not. If $i \geq 0$ and $i \neq 1$, then $S \setminus \{a^{-i}\}
    \subseteq \zln_a$ is coprincipal but not \mc.
    \item Let $R_1,R_2$ be rings and $u$ be a unit in $R_1$. Let
    $S \subseteq R_1 \times R_2$ be the subset $\{ (u^n,r) \mid n \geq 1 \}
    \cup \{(1,1)\}$. If $u^n \neq 1$ for any $n \geq 1$, or if $R_2$ is infinite,
    then $S$ is infinite, multiplicatively closed and coprincipal with
    cogenerator $(u,0)$.
  \end{enumerate}
\end{eg}

\begin{rmk}
  We have now seen that:
  \begin{itemize}[align=left,leftmargin=*,nosep]
    \item Coprincipal $\Rightarrow$ common multiplicatively closed
    \item If $S$ is finite, then $S$ is coprincipal $\Leftrightarrow$
    $S$ is common multiplicatively closed
    \item Multiplicatively closed $\Rightarrow$ common multiplicatively closed
    \item Coprincipal and multiplicatively closed do not imply or refute each other
    \item Common multiplicatively closed $\nRightarrow$ coprincipal
    \item Common multiplicatively closed $\nRightarrow$ multiplicatively closed
  \end{itemize}
\end{rmk}

\begin{cor}
  Let $R$ be a Noetherian ring and $S$ be a \cmc subset of $R$. The
  following are equivalent:
  \begin{enumerate}[label=(\alph*),align=left,leftmargin=*,nosep]
    \item $\tau_S$ is representable.
    \item $\tau_S$ is finitely generated.
    \item S is coprincipal.
  \end{enumerate}
\end{cor}

\begin{proof}
  First, we note that for all $M \in \bigmod(R)$, $\tau_S(M)
  = \varinjlim_{Rs} \bighom_R(R/(s),M) = \varinjlim_{Rs} (0:_M s)
  = \bigcup_s (0:_M s)$,
  where $s$ runs through $S$ and $(s) \geq (t)$
  iff $(s) \subseteq (t)$. So by Corollary~\ref{cor:dirlim},
  $\tau_S$ is finitely generated iff there exists $s \in S$ such that
  $\tau_S(M) = \bighom_R(R/(s),M)$ for all $M \in \fgmod(R)$ iff
  there exists $s \in S$ such that $(s) \subseteq (r)$ for all $r \in S$
  by considering $M = R/(r)$ for ``only if''.
\end{proof}

\begin{notn}
  We let $R^{\times}$ denote the set of units of a ring $R$.
\end{notn}

\begin{lem}
  Let $R$ be a ring and $S$ be a subset of $R$. Consider the following statements.
  \begin{enumerate}[label=(\alph*),align=left,leftmargin=*,nosep]
    \item S is coprincipal. \label{item:coprin}
    \item There are rings $R_1,R_2$ such that $R = R_1 \times R_2$, $S \cap
    (R_1)^{\times} \neq \emptyset$ and for all $s \in S$ we have
    $s(1,0) \in (R_1)^{\times}$.
    \label{item:idem}
  \end{enumerate}
  Then \ref{item:idem} $\Rightarrow$ \ref{item:coprin}. If $S$ is furthermore
  \mc, then \ref{item:coprin} $\Rightarrow$ \ref{item:idem}.
\end{lem}

\begin{proof}
  \ref{item:idem} $\Rightarrow$ \ref{item:coprin}: Let $(u,0) \in S \cap
  (R_1)^{\times}$ and $s \in S$. Since $s(1,0) \in (R_1)^{\times}$,
  $(u,0) \in Rs$. Therefore $(u,0)$ is a cogenerator of $S$.
  
  Now suppose that $S$ is \mc and coprincipal with cogenerator $e$.
  Since $S$ is \mc, $e^2 \in S$. Since $e$ is a cogenerator of $S$,
  $e=re^2$ for some $r \in R$. Then $(re)^2 = r(re^2) = re$, so $re$
  is idempotent. Let $R_1 = R(re)$ and $R_2 = R(1-re)$, so that
  $R = R_1 \times R_2$.
  Then $e(re) = re^2 = e$, so $e \in R_1$, and $e(r^2e) = (re)^2 = re$, so
  $e \in S \cap (R_1)^{\times}$. Finally, let $s \in S$. Then $e = r's$
  for some $r' \in R$, and $(r'r^2e)(sre) = (re)^3 = re$, so $sre
  \in (R_1)^{\times}$.
\end{proof}

\begin{lem} \label{lem:taus}
  Let $R$ be a ring, $S \subseteq R$ and $M \in \bigmod(R)$. If $\tau_S(M)$
  is a submodule of $M$, then $\ap_R(\tau_S(M)) = \{ P \in \ap_R(M) \mid
  P \cap S \neq \emptyset \}$. If $R$ is Noetherian and $S$ is a \mc
  subset of $R$, then $\ap_R(M/\tau_S(M)) = \{ P \in \ap_R(M) \mid
  P \cap S = \emptyset \}$.
\end{lem}

\begin{rmk}
  The second half of Lemma~\ref{lem:taus} is false if $S$ is not \mc.
  For example, let $R = \zln$, $S = \{p\}$ where $p$ is prime, and
  $M = \zln/(p^2)$. Then $\ap_R(M/\tau_S(M)) = \{(p)\}$, but $(p) \cap S
  \neq \emptyset$.
\end{rmk}

\begin{cor}
  Let $R$ be a Noetherian ring, $S$ be a \mc subset of $R$ and
  $\{M_n\}_{n \geq 0}$ be a sequence of modules in $\fgmod(R)$ such that
  the sets $\ap_R(M_n)$ stabilize. If $S$ is not coprincipal, then
  the functor $\id\!/\tau_S$ is finitely generated but not coherent,
  and $\tau_S$ is not finitely generated. However, whether or not
  $S$ is coprincipal, the sets $\ap_R(M_n/\tau_S(M_n))$
  and $\ap_R(\tau_S(M_n))$ always stabilize.
\end{cor}

\section{Covariant functors over a Dedekind domain}
\label{sec:dd}

In Section 2, we saw that the sets $\ap_R F(M/I^n M)$ stabilize whenever $F$ is
a coherent functor. One may ask whether such asymptotic stability still holds
when $F$ is not coherent. In this section, we consider the case where $R$ is
a Dedekind domain. We will see that if $F$ is a finitely generated functor over
$R$, then the sets $\ap_R F(M/I^n M)$ stabilize. We then construct a family of examples
of $R$-linear covariant functors $F$ such that the sets $\ap_R F(R/I^n)$ do not stabilize.

\begin{lem} \label{lem:ann}
  Let $R$ be a ring, $F$ be an $R$-linear functor from $\bigmod(R)$
  to itself and $M \in \bigmod(R)$. Then $\ann_R(M) \subseteq
  \ann_R(F(M))$.
\end{lem}

\begin{thm} \label{thm:dd}
  Let $R$ be a Dedekind domain, $I$ be an ideal of $R$, $M \in \fgmod(R)$
  and $F$ be a finitely generated functor. Then the sets $\ap_R F(M/I^n M)$
  stabilize.
\end{thm}

\begin{proof}
  The proof will proceed in several steps.
  \separate
  
  \textit{Step 1.} First, we will make some reductions. Since $F$ is
  $R$-linear, it preserves finite direct sums. By the structure theorem
  for finitely generated modules over a Dedekind domain, we may assume
  that $M = J$ is an ideal of $R$ or $M = R/P^i$ for some maximal ideal
  $P$ of $R$ and $i \geq 1$. If $M = R/P^i$, then either $M/I^n M = 0$
  for all $n$ or $M/I^n M = M$ for all $n \geq i$. If $0 \neq M = J
  \subseteq R$ and $I \neq 0$, then $M/I^n M \cong R/I^n$ for all $n \geq 1$.
  But $R/I^n$ is again a direct sum of modules of the form $R/P^{ni}$.
  So it suffices to show that asymptotic stability holds for $\ap_R F(R/P^n)$, where $P$
  is a maximal ideal of $R$. Furthermore, by Lemma~\ref{lem:ann},
  $\ap_R F(R/P^n) = \{ P \}$ or $\emptyset$ for all $n \geq 1$. So we only
  need to show that $F(R/P^n)$ is either always 0 or always nonzero
  for all large $n$.
  \separate
  
  \textit{Step 2.} Let $F$ be given by the surjection $h_L \to F$,
  where $L \in \fgmod(R)$. First we consider the case where $L=J$ is an
  ideal of $R$. Suppose that $F(R/P^n)=0$ for infinitely many $n$.
  We will show that in fact $F(R/P^n)=0$ for all $n$, which will
  conclude this case. So fix $n \geq 1$. Let $N \geq n$ be such that
  $F(R/P^N)=0$. Let $\pi \colon R/P^N \to R/P^n$ be the natural projection
  map. Since $J$ is a projective $R$-module, the map $h_J(\pi) \colon
  h_J(R/P^N) \to h_J(R/P^n)$ is surjective. From the commutative diagram
  \[
    \xymatrix{
      \bighom_R \left( J,\frac{R}{P^N} \right)
      \ar@{->>}[r] \ar@{->>}[d]^{h_J(\pi)}
      & F \left( \frac{R}{P^N} \right) = 0
      \ar[d]^{F(\pi)}\\
      \bighom_R \left( J,\frac{R}{P^n} \right)
      \ar@{->>}[r]
      & F \left( \frac{R}{P^n} \right)
    }
  \]
  we see that $F(\pi)$ is surjective and therefore $F(R/P^n)=0$.
  \separate
  
  \textit{Step 3.} Next, we consider the case where $L=R/Q^i$ such that $Q$
  is a maximal ideal of $R$ and $i \geq 1$. We may assume that $Q=P$.
  Suppose that $F(R/P^N)=0$ for some $N \geq i$. We will show that in fact
  $F(R/P^n)=0$ for all $n \geq N$, concluding this case. We recall the
  following facts. For any $n_1 \geq 1$, $R/P^{n_1}$ is a principal ideal ring.
  Choose an element $p \in P \setminus P^2$. Then $P^{n_2}/P^{n_1}$ is generated
  by $p^{n_2}$ for all $0 \leq n_2 \leq n_1$. Now fix $n \geq N$. Let
  $p^{n-N} \colon R/P^N \to R/P^n$ denote multiplication by $p^{n-N}$.
  Again from the commutative diagram
  \[
    \xymatrix{
      \frac{R}{P^i}
      \ar[rr]^-{\cong}_-{p^{N-i}} \ar[d]^{\id}
      && \frac{P^{N-i}}{P^N}=
      \bighom_R \left( \frac{R}{P^i},\frac{R}{P^N} \right)
      \ar@{->>}[r] \ar[d]^{p^{n-N}=h_J \left( p^{n-N} \right)}_{\cong}
      & F \left( \frac{R}{P^N} \right) = 0
      \ar[d]^{F \left( p^{n-N} \right)}\\
      \frac{R}{P^i}
      \ar[rr]^-{\cong}_-{p^{n-i}}
      && \frac{P^{n-i}}{P^n}=
      \bighom_R \left( \frac{R}{P^i},\frac{R}{P^n} \right)
      \ar@{->>}[r]
      & F \left( \frac{R}{P^n} \right)
    }
  \]
  we see that $F(p^{n-N})$ is surjective and therefore $F(R/P^n)=0$.
  \separate
  
  \textit{Step 4.} Finally, we consider the general case where
  $L = J_1 \oplus \cdots \oplus J_k \oplus R/Q_1^{i_1} \oplus \dots \oplus
  R/Q_{\ell}^{i_{\ell}}$ such that $J_1,\dots,J_k \subseteq R$ are ideals,
  $Q_1,\dots,Q_{\ell}$ are maximal ideals of $R$ and $i_1,\dots,i_{\ell}
  \geq 1$. Again we may assume that $Q_1=\dots=Q_{\ell}=P$. Suppose
  that $F(R/P^n)=0$ for infinitely many $n$. Fix $N \geq \max\{1,
  i_1,\dots,i_{\ell}\}$ such that $F(R/P^N)=0$. Then repeating Steps 2 and 3,
  we see that for all $n \geq N$, each direct summand of $h_L(R/P^n) =
  h_{J_1}(R/P^n) \oplus \dots \oplus h_{J_k}(R/P^n) \oplus h_{R/P^{i_1}}(R/P^n)
  \oplus \dots \oplus h_{R/P^{i_{\ell}}}(R/P^n)$ is mapped to 0 in $F(R/P^n)$.
  Therefore $F(R/P^n)=0$ for all $n \geq N$.
\end{proof}

\begin{lem}
  Let $R$ be a Dedekind domain, $I$ be an ideal of $R$ and $M \in \fgmod(R)$.
  Then the modules $I^n M / I^{n+1}M$ are all isomorphic for large $n$.
  In particular, let $F$ be an $R$-linear functor from $\bigmod(R)$ to itself.
  Then the sets $\ap_R F(I^n M/I^{n+1}M)$ stabilize.
\end{lem}

\begin{proof}
  As in Step 1 of Theorem~\ref{thm:dd}, we may assume that $M=J$ is an ideal
  of $R$ or $M=R/P^i$ for some maximal ideal $P$ of $R$ and $i \geq 1$. If
  $M=J \neq 0$ and $I \neq 0$, then $I^n M /I^{n+1}M \cong R/I$ for all
  $n \geq 0$. If $M = R/P^i$, then $I^n M / I^{n+1}M = 0$ for all $n \geq i$.
\end{proof}

\begin{thm} \label{thm:osc}
  Let $R$ be a Dedekind domain and $I \neq 0$ be an ideal of $R$.
  Then there exists $F \in \rlin$ such that the sets $\ap_R F(R/I^n)$
  do not stabilize. In fact, we may construct $F$ such that $\ap_R F(R/I^n)$
  is given by any sequence of subsets of $\ap_R (R/I) = V(I)$.
\end{thm}

\begin{proof}
  First, let $\tormod \subseteq \fgmod(R)$ be the subcategory of finitely
  generated torsion $R$-modules. Then the torsion functor $\tau \colon \fgmod(R)
  \to \tormod$ is $R$-linear. Next, we recall from category theory that
  any category is naturally equivalent to any skeleton of itself. In particular,
  given a skeleton $\tormod_0$ of $\tormod$, there is an $R$-linear functor
  $\pi \colon \tormod \to \tormod_0$. Therefore it suffices to construct
  $F \colon \tormod_0 \to \tormod_0$ as in our Theorem.
  \separate
  
  We will define $\tormod_0$ as follows. Fix a linear ordering $\preceq$
  of the nonzero prime ideals $R$, and let the objects of $\tormod_0$
  be modules of the form $R/P_1^{e_1} \oplus \dots \oplus R/P_j^{e_j}$,
  where $P_1 \preceq \dots \preceq P_j$ and $e_i \leq e_{i+1}$ whenever
  $P_i = P_{i+1}$. For each maximal ideal $P$ we choose a subset $S_P$ of
  $\nn_{>0}$. Then we define $F(R/P^e) = R/P$ if $e \in S_P$, and 0 otherwise.
  We let $F(R/P_1^{e_1} \oplus \dots \oplus R/P_j^{e_j}) = \oplus_{i=1}^k
  R/P$, where $k$ is the number of $F(R/P_i^{e_i})$ that are nonzero.
  Next we define $F(f)$ for $f \colon M \to N$, where $M,N \in \tormod_0$.
  It suffices to consider the case where $M,N$ are both $P$-torsion for some
  maximal ideal $P$ of $R$. Fix an element $p \in P \setminus P^2$. Then
  $\bighom_R(R/P^{n_1}, R/P^{n_2}) = P^{n_2-n_1}/P^{n_2}$ is generated by
  $p^{n_2-n_1}$ if $n_2 \geq n_1 \geq 1$, and $\bighom_R(R/P^{n_1},R/P^{n_2})=
  R/P^{n_2}$ if $n_1 \geq n_2 \geq 1$. So we can identify $f$ with a square matrix
  with entries in $R$ (more precisely, in $R/P^{e_i}$ for suitable $e_i$)
  viewed as multiplication maps, adding rows or columns of zeroes if necessary.
  If $M,N$ are both direct sums of copies of $R/P^{e_1},\dots,R/P^{e_j}$
  with $e_1 < \dots < e_j$, then we define
  \[
    F(f) = F \left(
    \begin{matrix}
      \begin{matrix}
        A_1 & \\
            & A_2
      \end{matrix}
      & \ast \\
      p\ast & \begin{matrix}
        \ddots & \\
               & A_j
      \end{matrix}
    \end{matrix}
    \right)
    = \left(
    \begin{matrix}
      \begin{matrix}
      A_1 & \\
          & A_2
      \end{matrix}
      & 0 \\
      0 & \begin{matrix}
        \ddots & \\
               & A_j
      \end{matrix}
    \end{matrix}
    \right),
  \]
  where the entries in the lower diagonal of the matrix on the left
  are multiples of $p$, and $A_1,A_2,\dots,A_j$ are the square blocks
  that correspond to $R/P^{e_1},\dots,R/P^{e_j}$ respectively. Since
  $F(R/P^e) =$ either $R/P$ or 0, the definition of $F(f)$ does not
  depend on the choice of coset representatives in the entries of $f$.
  It is then immediate that $F$ preserves identity maps and is $R$-linear.
  Finally, if $f \colon M \to N$ and $g \colon N \to L$ where $M,N,L$
  are $P$-torsion, then
  \begin{align*}
    F(g \circ f) &=
    F \left(
    \left(
      \begin{matrix}
        \begin{matrix}
           B_1 & \\
               & B_2
         \end{matrix}
               & \ast \\
           p\ast & \begin{matrix}
           \ddots & \\
                  & B_j
           \end{matrix}
      \end{matrix}
    \right)
    \left(
      \begin{matrix}
        \begin{matrix}
           A_1 & \\
               & A_2
         \end{matrix}
               & \ast \\
           p\ast & \begin{matrix}
           \ddots & \\
                  & A_j
           \end{matrix}
      \end{matrix}
    \right)
    \right)\\
    &=
    F \left(
      \begin{matrix}
        \begin{matrix}
           B_1 A_1 + p\ast & \\
                           & B_2 A_2 + p\ast
         \end{matrix}
               & \ast \\
           p\ast & \begin{matrix}
           \ddots & \\
                  & B_j A_j + p\ast
           \end{matrix}
      \end{matrix}
    \right)\\
    &=
    \left(
      \begin{matrix}
        \begin{matrix}
           B_1 A_1 + p\ast & \\
                           & B_2 A_2 + p\ast
         \end{matrix}
               & 0 \\
           0 & \begin{matrix}
           \ddots & \\
                  & B_j A_j + p\ast
           \end{matrix}
      \end{matrix}
    \right)\\
    &=
    \left(
      \begin{matrix}
        \begin{matrix}
           B_1 A_1 & \\
                   & B_2 A_2
         \end{matrix}
               & 0 \\
           0 & \begin{matrix}
           \ddots & \\
                  & B_j A_j
           \end{matrix}
      \end{matrix}
    \right)
    = F(g)F(f)
  \end{align*}
  Therefore $F$ respects composition.
\end{proof}

\begin{cor}
  The functors constructed in Theorem~\ref{thm:osc} are not finitely
  generated.
\end{cor}

\begin{qn}
  Is there a finitely generated non-coherent functor $F$ such that the sets
  $\ap_R F(R/I^n)$ do not stabilize?
\end{qn}

\section{Functors arising from middle finite complexes}
\label{sec:midfin}

In this section, we will study a class of $R$-linear covariant
functors $F$ which arise naturally and are non-finitely generated in general.
An example of such kind of functor is the zeroth
local cohomology functor. We will obtain results that are related to all the
previous sections. Our main result is that over a one-dimensional Noetherian
domain $R$, the sets $\ap_R F(M/I^n M)$ stabilize.

\begin{defn}
  Let $R$ be a ring and $\seq \colon A \to B \to C$ be a complex of $R$-modules.
  \begin{enumerate}[label=(\arabic*),align=left,leftmargin=*,nosep]
    \item We say that an $R$-linear functor $F \colon \bigmod(R) \to
    \bigmod(R)$ arises from $\seq$ if $F(-)=\bigh(\seq \otimes -)$.
    \item We say that $\seq$ is middle finite if $B \in \fgmod(R)$.
  \end{enumerate}
\end{defn}

\begin{eg}
  Let $R$ be a ring and $I=(x_1,\dots,x_n)$ be an ideal of $R$. Then the
  functor $\Gamma_I$ arises from the middle finite complex
  \[
    0 \to R \to R_{x_1} \oplus \dots \oplus R_{x_n}
  \]
\end{eg}

\begin{rmk}
  Let $R$ be a Noetherian ring. By Corollary~\ref{cor:lc}, a functor that
  arises from a middle finite complex of $R$-modules is not finitely generated
  in general.
\end{rmk}

\begin{lem}
  Let $R$ be a Noetherian ring. Let $F$ be a functor that arises from the
  middle finite complex $A \xrightarrow{\partial_A} B \xrightarrow{\partial_B}
  C$. Then $F$ is coherent iff it is finitely generated.
\end{lem}

\begin{proof}
  Suppose that $F$ is finitely generated and is given by the surjection
  $h_M \to F$. Let $K,I$ denote the functors given by $K(-)=\ker(\partial_B
  \otimes -)$ and $I(-)= \im(\partial_A \otimes -)$. Let $N \in \fgmod(R)$
  and $n \in K(N)$ be such that $n + I(N) \in F(N)$. By Lemma~\ref{lem:Yon},
  there is $f \in \bighom_R(M,N)$ such that $n + I(N)\in \im F(f)$. That is,
  there are $m \in K(M)$ and $x \in A \otimes N$ such that $n =
  (\id_B \otimes f)(m) + (\partial_A \otimes \id_N)(x)$.
  Now $C \otimes M = \varinjlim_D (D \otimes M)$, where
  $D$ ranges over all finitely generated submodules of $C$. Since $B \otimes M
  \in \fgmod(R)$, there is a finitely generated submodule $C_0$ of $C$ that
  contains $\im \partial_B$ such that $\ker(B \otimes M \to
  C \otimes M) = \ker(B \otimes M \to C_0 \otimes M)$. From the commutative
  diagram
  \[
    \xymatrix@C=54pt{
      A \otimes M \ar[r]^{\partial_A \otimes \id_M} \ar[d]^{\id_A \otimes f}
      & B \otimes M \ar[r]^{\partial_B \otimes \id_M} \ar[d]^{\id_B \otimes f}
      & C_0 \otimes M \ar[d]^{\id_{C_0} \otimes f}\\
      A \otimes N \ar[r]^{\partial_A \otimes \id_N}
      & B \otimes N \ar[r]^{\partial_B \otimes \id_N}
      & C_0 \otimes N
    }
  \]
  we see that in fact $n \in \ker(B \otimes N \to C_0 \otimes N)$. Finally, let
  $A_0$ be a finitely generated submodule of $A$ such that $\partial_A(A_0)
  = \partial_A(A)$. Then $F$ arises from the complex $A_0 \to B \to C_0$.
  Therefore $F$ is coherent by Lemma~\ref{lem:Har}.
\end{proof}

\begin{lem}
  Let $R$ be a Noetherian ring, $I,J$ be ideals of $R$, $M \in \fgmod(R)$,
  $M'$ be a submodule of $M$ and $F$ be a functor that arises from the
  middle finite complex $A \to B \to C$. Then the sets $\ap_R F(I^n M/I^n M')$
  and the values $\dep_J F(I^n M/I^n M')$ stablize.
\end{lem}

\begin{proof}
  The module $\bigoplus_{n \geq 0} B \otimes (I^n M/I^n M')$ is finitely
  generated and graded over $S = \bigoplus_{n \geq 0} I^n$, and the maps
  in the induced complex
  \[
    \bigoplus_{n \geq 0} A \otimes \frac{I^n M}{I^n M'}
    \to \bigoplus_{n \geq 0} B \otimes \frac{I^n M}{I^n M'}
    \to \bigoplus_{n \geq 0} C \otimes \frac{I^n M}{I^n M'}
  \]
  are homogenous of degree 0. The result then follows from
  Corollaries~\ref{cor:gradedap} and \ref{cor:gradeddep}.
\end{proof}

\begin{thm} \label{thm:midfin}
  Let $R$ be a one-dimensional Noetherian domain, $I$ be an ideal of $R$,
  $M \in \fgmod(R)$ and $F$ be a functor that arises from the middle finite
  complex $\seq \colon A \xrightarrow{\alpha} B \xrightarrow{\beta} C$.
  Then the sets $\ap_R F(M/I^n M)$ stabilize.
\end{thm}

\begin{proof}
  First, since $\seq \otimes (M/I^n M) = (\seq \otimes M) \otimes (R/I^n)$,
  it suffices to show that the sets $\ap_R F(R/I^n)$ stabilize. We have
  $\seq \otimes (R/I^n) \colon A/I^n A \xrightarrow{\alpha_{n-1}} B/I^n B
  \xrightarrow{\beta_{n-1}} C/I^n C$, so
  \[
    F(R/I^n) = \frac{\ker \beta_{n-1}}{\im \alpha_{n-1}}
    = \frac{\beta^{-1}(I^n C)}{\alpha(A) + I^n B}
    = F'(R/I^n),
  \]
  where $F'$ arises from the complex $0 \to B/\alpha(A) \to C$. So we may
  assume that $A=0$. Furthermore, since localization is flat, we may assume
  that $R$ is local of dimension one. So it remains to show that $F(R/I^n)$
  is either always 0 or always nonzero for all large $n$.
  \separate
  
  Now let $S = \bigoplus_{n \geq 0} I^n$ and $\gamma \colon \bigoplus_{n \geq 0}
  (I^n B / I^{n+1} B) \to \bigoplus_{n \geq 0} (I^n C / I^{n+1} C)$ be the map
  induced by $\beta$ with graded components $\gamma_n$. By
  Corollary~\ref{cor:gradedap}, there is $N$ so large such that the sets
  $\ap_R (\ker \gamma_n)$ are equal for all $n > N$. Again we have $\beta_n
  \colon B/I^{n+1} B \to C/I^{n+1} C$, so that $F(R/I^n) = \ker \beta_{n-1}$.
  Suppose that there is $m > N$ such that $\ker \beta_{m-1} = 0$ but
  $\ker \beta_m \neq 0$. Then $I^{m+1} B \subsetneq \beta^{-1}(I^{m+1} C)
  \subseteq \beta^{-1} (I^m C) = I^m B$, so that $0 \neq \ker \beta_m
  \subseteq \ker \gamma_m$, and hence $\ker \gamma_n \neq 0$ for all $n > N$.
  But $\ker \beta_n \supseteq \ker \gamma_n$ always holds. Therefore
  we have $\ker \beta_n \neq 0$ for all $n>N$.
\end{proof}


\begin{thebibliography}{99}

\bibitem{BPrime} M.\ Brodmann,
\textit{Asymptotic Stability of $\ap (M/I^n M)$},
Proc.\ Amer.\ Math.\ Soc., \textbf{74}
(1979), 16--18.

\bibitem{BSpread} --,
\textit{The Asymptotic Nature of the Analytic Spread},
Math.\ Proc.\ Camb.\ Phil.\ Soc.,
\textbf{86} (1979), 35--39.

\bibitem{Har} R.\ Hartshorne,
\textit{Coherent Functors},
Advances in Mathematics \textbf{140} (1998), 44--94.


\bibitem{KW} D.\ Katz, E.\ West,
\textit{A Linear Function Associated
to Asymptotic Prime Divisors},
Proc.\ Amer.\ Math.\ Soc.,
\textbf{132} (2003), no.\ 6, 1589--1597.

\bibitem{McA} S.\ McAdam,
\textit{Asymptotic Prime Divisors},
Lecture Notes in Mathematics (Springer-Verlag) 1023;
Springer-Verlag 1983.

\bibitem{MS} L.\ Melkersson,
P.\ Schenzel, \textit{Asymptotic Prime
Ideals Related to Derived Functors},
Proc.\ Amer.\ Math.\ Soc., \textbf{117}
(1993), no.\ 4, 935--938.

\bibitem{R} L.~J.~Ratliff, Jr.,
\textit{On Prime Divisors of $I^n$,
$n$ Large}, Michigan Math.\ J., \textbf{23}
(1976), 337--352.

\bibitem{Si} A.\ K.\ Singh,
\textit{$p$-torsion Elements in Local Cohomology Modules},
Mathematical Research Letters \textbf{7}
(2000), 165--176.

\bibitem{SS} A.\ K.\ Singh,
I.\ Swanson, \textit{Associated Primes
of Local Cohomology Modules and of
Frobenius Powers}, IMRN No.\ 33 (2004),
1703--1733.
\end{thebibliography}
\end{document}